\documentclass[10pt]{amsart}
\usepackage{latexsym, amsmath,amssymb}
 \numberwithin{equation}{section}

\newcommand{\lcal}{\mathcal{L}}

\newtheorem{theo}{{\sc Theorem}}[section]

\newtheorem{defin}{{\sc Definition}}

\newtheorem{lem}[theo]{{\sc Lemma}}

\newenvironment{rem}{\medskip\noindent{\it Remark:\/} }{\medskip}

\renewcommand{\epsilon}{\varepsilon}
\newtheorem{theorem}{Theorem}

\newtheorem{lemma}[theorem]{Lemma}
\newtheorem{corr}[theorem]{Corollary}
\newtheorem{prop}[theorem]{Prop}
\newtheorem{proposition}[theorem]{Proposition}
\newtheorem{deff}[theorem]{Definition}

\newcommand{\bth}{\begin{theorem}}
\newcommand{\ble}{\begin{lemma}}
\newcommand{\bcor}{\begin{corr}}

\newcommand{\bdeff}{\begin{deff}}

\newcommand{\bprop}{\begin{proposition}}
\newcommand{\ele}{\end{lemma}}
\newcommand{\ecor}{\end{corr}}
\newcommand{\edeff}{\end{deff}}

\numberwithin{theorem}{section}

\newcommand{\eprop}{\end{proposition}}

\newcommand{\la}{\lambda}

\renewcommand{\Pi}{\varPi}

\renewcommand{\epsilon}{\varepsilon}

\newcommand{\R}{{\mathbb R}}
\newcommand{\Z}{{\mathbb Z}}

\begin{document}

\title[Focal points  II: the two-dimensional case]
{Focal points and sup-norms of eigenfunctions on analytic Riemannian
manifolds II: the two-dimensional case
}

\author[C. D. Sogge]{Christopher D. Sogge}
\address{Department of Mathematics, Johns Hopkins University, Baltimore, MD 21218, USA}
\email{sogge@jhu.edu}
\author{Steve Zelditch}
\address{Department of Mathematics, Northwestern University, Evanston, IL 60208, USA}
\email{s-zelditch@northwestern.edu}

\thanks{Th research of the first author was partially supported by NSF grant \# DMS-1361476 and the second by  \#  DMS-1206527.  
}

\begin{abstract} We use a purely dynamical argument on circle maps
to improve a recent result on real analytic surfaces possessing eigenfunctions
that achieve maximal sup norm bounds. The improved result is that there
exists a `pole'  $p$ so that all  geodesics emanating from $p$ are smoothly
closed .
\end{abstract}

\maketitle

\section{Introduction and main results}

In a recent article \cite{SZRA}, the authors gave a dynamical characterization
of compact real analytic Riemannian manifolds $(M^n, g)$ of dimension n possessing
$\Delta_g$-eigenfunctions 
$$(\Delta + \lambda^2_{j_k}) e_{j_l} = 0, \;\;\; ||e_{j_k}||_{L^2} = 1$$
of maximal sup norm growth,
 \begin{equation} \label{MAX}  \|e_j\|_{L^\infty(M)} \geq C_g \; \la_j^{\frac{n-1}{2}}. \end{equation}
Here, $C_g$ is a positive constant independent of $\lambda_j$. 
The main result of \cite{SZRA} (recalled more precisely below) is that if  $(M, g)$ 
possesses such a sequence $\{e_{j_k}\}$, then there must exist {\it self-focal points} $p$ at which all geodesics
from $p$ loop back to $p$ at some time. The minimal such time is called the first return time $T_p$. Moreover,
there must exist a self-focal point for which
the first return map $\eta_p: S^*_pM \to S^*_p M$ preserves an $L^1$ measure on the unit co-sphere $S^*_p M$ at p.   The purpose of this addendum is to add a purely dynamical
argument to the  main result of \cite{SZRA} to prove 
the stronger

\begin{theorem}\label{maintheorem}  Let $(M,g)$ be a compact real analytic compact surface without boundary.  If there exists a sequence of 
$L^2$-normalized eigenfunctions,
$(\Delta + \lambda^2_{j_k}) e_{j_k} = 0 $, 
satisfying $  \|e_{j_k}\|_{L^\infty(M)} \geq C_g \; \la_{j_k}^{\frac{1}{2}},$ then $(M, g)$
possesses a pole, i.e. a point $p$  so that every geodesic starting at $p$
returns to p at time $2 T_p$ as a smoothly closed geodesic. 
\end{theorem}

Thus, $(M, g)$ is a $C^p_{2T_p}$-manifold in the terminology of \cite{Be}
(Definition 7.7(e)). 
Theorem \ref{maintheorem} proves the  conjecture stated on p. 152 of \cite{STZ} in the case of real analytic
surfaces.  It  follows by combining the main result of
\cite{SZRA} with the following:

\begin{prop} \label{etap} Let $(S^2, g)$ be a two-dimensional real analytic Riemannian
surface. Suppose that $p \in S^2$ is a  self-focal point and that
the first return map $\eta_p : S^*_p S^2 \to S^*_p S^2$ preserves a probability
measure which is in $L^1(S^*_p S^2)$. Then $\eta_p^2$ is the identity map,
and in particular all geodesics through $p$ are smoothly closed with the common 
period  $2 T_p$.
\end{prop}

The proof only involves dynamics and not eigenfunctiosn of $\Delta_g$.
To explain how Theorem \ref{maintheorem} follows from Proposition
\ref{etap}, we  first recall some of the definitions and main result of \cite{SZRA} to establish
notation. We then give the proof of Proposition \ref{etap}. 

A natural question is whether all  real analytic Riemannian surfaces with
maximal eigenfunction growth  
are surfaces of revolution. There are many $P^m_{\ell}$ metrics
besides surfaces of revolution.  A second question is whether the Proposition
has some kind of generalization to higher dimensions. 

We thank Keith Burns for reading an earlier version of this note and
for his comments. 

\section{Background on maximal eigenfunction growth}

Let 
$\eta_t(x,\xi)=(x(t),\xi(t))$ denote the  homogeneous Hamilton flow on $T^* M \backslash 0$  generated by  $H(x, \xi) = |\xi|_g$.  Since  $\eta_t$ 
preserves the unit cosphere bundle $S^*M = \{|\xi|_g = 1\}$, it defines a flow on  $S^*M$ which preserves   Liouville measure.  For a given $x\in M$, let ${\mathcal L}_x\subset S^*_xM$ denote the set of {\it loop directions}, i.e.  unit directions $\xi$ for which $\eta_t(x,\xi)\in S_x^*M$ for some time $t\ne 0$. Also, let
$d\mu_x$ denote the  measure on  $S^*_x M$  induced by the Euclidean metric $g_x$ and let $|{\mathcal L}_x| = \mu_x({\mathcal L}_x)$.

We say that $p$ is a {\it self-focal point}  if there exists a  time $\ell>0$ so that $\eta_\ell(p,\xi)\in S^*_pM$ for
all $\xi\in S^*_p M$, i.e. if $\lcal_p = S^*_p M$.   We let $T_p$ be the minimal such time,  and write
\begin{equation} \eta_{T_p}(p,\xi)=(p,\eta_p(\xi)), \quad \xi\in S^*_pM.
\end{equation}Under the assumption that $g$ is real analytic, the {\it first
return map}
\begin{equation} \label{FRM} \eta_p: S^*_pM\to S^*_pM \end{equation}
 is also   real analytic.  

The key property of interest is that $\eta_p$ is conservative in the following
sense:

\begin{defin} We say  $\nu_p$ is conservative if it preserves a measure $\rho d\mu_p$
on $S^*_p M$ which is absolutely continuous with respect to $\mu_p$
and of finite mass:
\begin{equation} \label{PROP} \exists \rho     \in L^1(\mu_p), \;\; \eta_p^* 
\rho d\mu_p = \rho d\mu_p. \end{equation}
\end{defin}

\begin{theo} \cite{SZRA} Suppose that $(M, g)$ is a compact  real analytic manifold without
boundary which  possesses a sequence
$\{e_{j_k}\}$ of eigenfunctions satisfying \eqref{MAX}. Then $(M, g)$
possesses a self-focal point $p$ whose first return map $\eta_p$ is 
conservative.
\end{theo}

\section{Proof of Proposition \ref{etap}}

It is not particularly important to the proof, but we may assume with
no loss of generality that $M$ is diffeomorphic to $S^2$. The proof
is a standard one on manifolds with focal points; we refer to \cite{SZ}
for background. We also assume throughout that $\eta_p$ is real 
analytic, since that is the case in our setting; most of the statements
below are true for smooth circle maps.

We note that $\eta_p$ is the restriction of the geodesic flow  $G^{T}$ to the invariant
set $S^*_p S^2$. This circe is contained in a symplectic transversal $S_p$ to the geodesic
flow. On the symplectic transversal $G^T$ is a symplectic map  which is invertible. Hence
$\eta_p$ is invertible. 
Thus, $(D \eta_p)_{\omega}$ is non-zero for all $\omega \in S^*_p S^2$.  It follows that $\eta_p$ is
either orientation preserving or orientation reversing.

Next we use  time reversal invariance of the geodesic
flow to show that $\eta_p$ is conjugate to its inverse.

\begin{lem} $\eta_p$ is reversible (conjugate to its inverse). \end{lem}

\begin{proof} Let $\tau(x, \xi) = (x, -\xi)$ on $S^*M$. Then on all of $S^*M$, we have  $\tau G^t \tau = G^{-t}$. 
Indeed, let  $\Xi_H$ be the Hamilton vector field of $H(x, \xi) = |\xi|_g$.
Then $H \circ \tau = H$, i.e. $H$ is time reversal invariant. We
claim that $\tau_* \Xi_H = - \Xi_H$. Written in Darboux
coordinates,
$$\Xi_H = \sum_j \frac{\partial H}{\partial \xi_j}
\frac{\partial}{\partial x_j} - \frac{\partial H}{\partial x_j}
\frac{\partial}{\partial \xi_j}. $$ If we let $(x, \xi) \to (x, -
\xi)$ and use invariance of the Hamiltonian we see that the vector
field changes sign.
Now, $G^{-t}$ is the Hamilton flow of $- \Xi_H$ and that is
$\tau_* \Xi_H$. But the Hamilton flow of the latter is $\tau G^t
\tau$.

Since $S^*_p M$ is invariant under $\tau$,  we just
restrict the identity  $\tau G^t \tau = G^{-t}$ to $S^*_p M $ to see that $\eta_p$ is reversible.
\end{proof}

\subsection{Orientation preserving case}

First let us assume that $\eta_p$ is orientation preserving. Then it has a rotation number. 
We recall that the rotation number of a circle homeomorphism is
defined by
$$r(f) = \left( \lim_{n \to \infty} \frac{F^{n}(x) - x}{n} \right)
\;\;\; \mbox{mod}\;\; 1. $$ Here, $F:\R \to \R$ is a lift of $f$,
i.e. a  map satisfying $F(x + 1) = F(x)$ and $f = \pi \circ F$
whee $\pi: \R \to \R/\Z$ is the standard projection. The rotation
number is independent of the choice of $F$ or of $x$. It is
rational if and only if $f$ has a periodic orbit. For background,
see \cite{F}.

\begin{lem} The rotation number of $\eta_p$ is either $0$ or $\pi$.
\end{lem}

\begin{proof}

For a circle homeomorphism, the
rotation number $\tau(f^{-1})$ is always $- \tau(f)$. Since $\eta_p$
is reversible, $\tau(\eta_p) = - \tau(\eta_p)$, i.e. its   rotation number can only be $0, \pi$. 

\end{proof}

\begin{lem} $\eta_p^2$ has fixed points. \end{lem}

\begin{proof}
The rotation number of $\eta_p^2$ is $0$. But
it  is known that  $\tau(f) =
0$ if and only if $f$ has a fixed point. See \cite{F}, Theorem
2.4.
\end{proof}

We now complete the proof 
that $\eta_p^2 = Id$ if $\eta_p$ is orientable. Since $\eta_p$ is real
analytic, this is the case if  $\eta_p^2$ has infinitely
many fixed points, so we may assume that $\mbox{Fix}(\eta_p^2)$ is finite (and non-empty). We write
$\# \mbox{Fix}(\eta_p^2) = N$ and denote the fixed points by $p_j$.

If $N = 1$, i.e.   $\eta_p^2$ has one fixed point $Q$, then $S^1 \backslash \{Q\}$ is an interval
and $\eta_p^2$ is a monotone map of this interval. So every orbit is asymptotic to the fixed
point of $\eta_p^2$.

Let $\mu$ be the $L^1$  invariant measure for $\eta_p^2$ and let $K = \mbox{supp} \mu$. We
can decompose $K$ into $N$ subsets $K_j$ such that $\eta_p^2(K_j) \to p_j$.  $K_j$
is the basin of attraction of $p_j$.

 Then $$\mu(K_j) = \mu(\eta^{2p}(K)_j) \to \mu(\{p_j\}). $$
But $p_j \in K_j$ so it must be that $K_j = \{p_j\}$. This shows that $\mu$ cannot be $L^1$,
concluding the proof.

\subsection{Orientation reversing case}

The square $\eta_p^2$ of an orientation reversing diffeomorphism of $S^*_p S^2$ is
an orientation preserving diffeomorphism. If $\eta_p$ preserves the measure $d\mu$ then
so does $\eta_p^2$. Thus we reduce to the orientation preserving case.

\begin{rem} Keith Burns pointed out to us that the orientation reversing
case cannot occur. If $\eta_p$ is orientation reversing, so is $- \eta_p$. 
An orientation reversing homeomorphism of the circle must have a
fixed point $\xi \in S^*_p M$. But then $G^T(p, \xi) = (p, - \xi)$ for
some $T > 0$, which is impossible.

\end{rem}

\end{document}